\documentclass[11pt]{article}
\usepackage{color,xcolor}
\usepackage{graphicx}
\usepackage{amssymb}
\usepackage{amsthm}
\usepackage{amsmath,numbysec}
\usepackage[margin=1in]{geometry}
\usepackage{caption}
\usepackage{subcaption}
\usepackage{float}
\usepackage{mdframed,pdfsync, mathrsfs}

\newtheorem{thm}{Theorem}[section]
\newtheorem{lemma}[theorem]{Lemma}

\newtheorem{remark}[theorem]{Remark}

\newtheorem{definition}[theorem]{Definition}

\newcommand{\bea}{\begin{eqnarray*}}
\newcommand{\eea}{\end{eqnarray*}}
\newcommand{\ben}{\begin{eqnarray}}
\newcommand{\een}{\end{eqnarray}}
\newcommand{\beq}{\begin{equation}}
\newcommand{\eeq}{\end{equation}}

\newcommand{\R}{\ensuremath{\mathbb{R}}}

\renewcommand{\d}{\partial}

\newcommand{\grad}{\mathrm{grad}}
\newcommand{\Hess}{\mathrm{Hess}}

\begin{document}

\title{On the gradient flow structure of the isotropic Landau equation}
\date{}
\author{
  Jing An%
  \thanks{Institute for Computational and Mathematical Engineering, Stanford University, Stanford, CA 94305. 
    Email: {\tt jingan@stanford.edu}},~~
  Lexing Ying%
  \thanks{Department of Mathematics and ICME, Stanford University, Stanford, CA 94305.
    Email: {\tt lexing@stanford.edu}}
}

\maketitle

\numberbysection

\begin{abstract}
  We prove that the isotropic Landau equation equipped with the Coulomb potential introduced by
  Krieger-Strain and Gualdani-Guillen can be identified with the gradient flow of the entropy in the
  probability space with respect to a Riemannian metric tensor with nonlocal mobility. We give
  characterizations of the corresponding geodesics equations and present a convergence rate result
  by estimating its Hessian operator.
\end{abstract}

\section{Introduction}
Since Otto's pioneering work on analyzing the porous medium equation \cite{otto2001geometry}, there
has been a lot of work on exploring gradient flow structures of different partial differential
equations in the space of probability measures. The gradient flow method has proven to be important
for both analytical and numerical simulation purposes, for example \cite{jordan1998variational, benamou2000computational, santambrogio2017euclidean, carrillo2003kinetic, carrillo2017numerical, liu2016stein}, just to name a few .

Although there has been a vast amount of literature on mathematical analysis of Boltzmann and Landau
equations, the investigation of their gradient flow structures has only started quite recently: for
example see \cite{erbar2016gradient,basile2017gradient} for recent analysis results on the Boltzmann
equation and \cite{carrillo2019} for a novel numerical method on approximating the homogeneous
Landau equation. Very recently, Carrillo et al. \cite{carrillo2020} carry out in-depth gradient flow analysis of the homogeneous Landau equation and provide the theoretical basis of the $\epsilon-$approximated Landau equation that \cite{carrillo2019} aims to solve. As all these results are based on the dissipation of the entropy functional, i.e.,
H-theorems, those kinetic equations can be viewed as gradient flows of the entropy with respect to
various specific geometries.

The reason of limited progress on the gradient flow approach of the Boltzmann-like kinetic equations
is the following: unlike the classical $L^2$ Wasserstein distance and generalized Wasserstein
distance regarding concave, nonlinear, and local mobilities (see for example
\cite{carrillo2010nonlinear,liero2013gradient,li2018geometry,Li2019,li2019diffusion,garbuno2019gradient}),
the metrics associated with these Boltzmann-like kinetic equations involve nonlocal mobilities,
which cause significant challenges when one tries to analyze related displacement convexity,
functional inequalities, contractions and so on.

The goal of this short note is to identify the gradient flow structure of a modified version of
classical Landau-Coulomb equations, the isotropic Landau equation, which has drawn interests within
the kinetic community recently
\cite{krieger2012global,gressman2012non,gualdani2016estimates,gualdani2018global,gualdani2018review}. Because the isotropic Landau equation ignores the projection matrix, some of its properties are different from the original Landau equation. This also makes our note distinct from other papers on the gradient flow structure of the classical Landau equation. Given
this gradient flow structure, we are able to characterize some basic geometric properties and
calculate a time-dependent convergence rate for the entropy dissipation.

\subsection{The isotropic Landau equation }

Let us first recall the homogeneous Landau equation with the Coulomb potential
\begin{align}
  \d_t \rho = \nabla \cdot (A[\rho] \nabla \rho - \rho \nabla \mathcal{L}\rho), ~~\text{with}~
  A[\rho] = \frac{1}{8\pi|x|}\bigg(I_d - \frac{x\otimes x}{|x|^2}\bigg) *\rho.
\end{align}
The modified Landau equation, which shares structural similarities with Landau equation from plasma
physics and was first considered by Krieger and Strain \cite{krieger2012global}, has the form
\begin{align}
    \d_t \rho = \mathcal{L}\rho \Delta \rho + \alpha \rho^2, ~~\text{with}~ \mathcal{L}\rho = (-\Delta)^{-1}\rho.
\end{align}
So far its global-in-time well-posedness with radial monotonic positive initial data has been proven
in \cite{gressman2012non} for $\alpha\in (0,74/75)$. When $\alpha = 1$, the above can be rewritten
as
\begin{align}\label{landau}
\d_t \rho = \nabla \cdot (\mathcal{L}\rho \nabla \rho - \rho \nabla \mathcal{L}\rho),
\end{align}
which is called the isotropic Landau equation since $A[\rho]$ is replaced by $\mathcal{L}\rho$ and
has been studied by Gualdani and her collaborators in papers
\cite{gualdani2016estimates,gualdani2018global,gualdani2018review}. We recall that the inverse
fractional Laplacian operator $(-\Delta)^{-s}$ is a Riesz potential of order $2s$ and can be
expressed as
\begin{align}
  (-\Delta)^{-s} \rho(x,t) :=c_{d,s}\int_{\R^d} \frac{\rho(y,t)}{|x-y|^{d-2s}} dy, ~~ t>0,
\end{align}
with $c_{d,s} = \frac{4^s \Gamma(d/2+s)}{|\Gamma(-s)|\pi^{d/2}}$. Here for the isotropic Landau
equation (\ref{landau}), we only consider the case where $d=3$ and $s=1$, thus $c_{3,1} =
\frac{1}{4\pi}$.

\subsection{Previous work on the isotropic Landau equation}

The global-in-time existence of smooth solutions given radially symmetric and monotonically
decreasing initial data that have finite mass, energy and entropy was shown in
\cite{gualdani2016estimates}. Later, the radial symmetry requirement was relaxed to even functions
in \cite{gualdani2018global}. Although the isotropic Landau equation (\ref{landau}) is structurally
similar to the classical Landau-Coulomb equation, in analysis it is very different in the sense that
its second moment increases in time as mentioned in \cite{gualdani2018review, gualdani2018global}. Because of that, many techniques in the classical Landau equation do not directly apply and the
dissipation computation
\begin{align}\label{dissp}
  -\frac{d}{dt} \mathcal{E}(\rho) =\frac{1}{8\pi} \int_{\R^3}\int_{\R^3}\frac{\rho(x)\rho(y)}{|x-y|}
  \bigg|\frac{\nabla \rho(x)}{\rho(x)}-\frac{\nabla \rho(y)}{\rho(y)}\bigg|^2 dxdy \geq 0
\end{align}
does not imply a Maxwellian equilibrium. In fact, the only steady solution for the isotropic Landau
equation is the identically zero solution. 

Let us summarize some conditional regularity results from
\cite{gualdani2018global,gualdani2018review}, which will be used to analyze a distance based on
the nonlocal mobility to be introduced below. By assuming that the initial data $\rho_0$ is even,
$|\mathcal{E}(\rho_0)|<+\infty$ and $||\rho_0||_{L^1} = 1$ by normalization, the following
time-dependent dissipation-Fisher information relation holds
\begin{align}\label{prevrel}
  \frac{d\mathcal{E}(\rho)}{dt} + &\kappa(t) \int_{\R^3} \frac{|\nabla \sqrt{\rho(x,t)}|^2}{1+|x|} dx \leq 0, ~~ t>0, \\
  \nonumber    &~\text{with}~\kappa =  \frac{1}{8\pi}\frac{1}{E(t)^{1/2} +1} ~\text{and}~ E(t) = \int_{\R^3} \frac{|x|^2}{2} \rho(x,t) dx.
\end{align}
In addition, the second moment of data is locally bounded in time
\begin{align*}
  E(t) \leq C_{p,\epsilon} (1+t^{2p/(2p-4+\epsilon)}),~~ t>0, ~~\frac{9}{5}<p<2, ~~4-2p<\epsilon<\frac{2}{5}.
\end{align*}
Integrating \eqref{prevrel} in time gives rise to
\begin{align}
  \int_0^T \int_{\R^3} \frac{|\nabla \sqrt{\rho}|^2}{1+|x|} dxdt \leq C_T.
\end{align}
The conditional regularity estimates in \cite{gualdani2018global, gualdani2018review} are built on
assuming the following $\varepsilon-$Poincar{\'e} inequality
\begin{align}\label{poincare}
  \int_{\R^3} \rho \phi^2  dx\leq \varepsilon \int_{\R^3} (-\Delta)^{-1} \rho |\nabla \phi|^2 dx +
  C_{\varepsilon} \int_{\R^3} \phi^2 dx, ~~\text{ for } \phi\in L_{loc}^1(\R^3),
\end{align} 

The above assumptions allows for a uniform bound in space and time for $(-\Delta)^{-1}\rho(x,t)$ and
we restate this result here.
\begin{lemma}[\cite{gualdani2018review}, Theorem 2]
  Suppose $\rho$ is a solution to the isotropic Landau equation (\ref{landau}) with even
  non-negative initial data $\rho_0$, and (\ref{poincare}) holds. For any $0<t<T$ and any $s_1>1,
  s_2>1/3$, and any ball $B_R\subset \R^3$ with arbitrary radius $R>0$, there exists constants
  $C_1(T,R,s_1), C_2(T, s_2)$ such that
  \begin{align}
    &||\rho||_{L^{\infty}(t,T;B_R)} \leq C_1(T, R,s_1) \bigg(\frac{1}{t} + 1\bigg)^{s_1}, ~~t\in(0,T),\label{fracbd0}\\
    &||(-\Delta)^{-1}\rho||_{L^{\infty}(t,T;\R^3)} \leq C_2(T, s_2) \bigg(\frac{1}{t} + 1\bigg)^{s_2}, ~~t\in(0,T).
    \label{fracbd}
  \end{align}
\end{lemma}

\subsection{Main results}

The first main result of this note is the following gradient flow characterization of the isotropic
Landau equation.
\begin{thm}
  The isotropic Landau equation (\ref{landau}) can be viewed as the gradient flow for the Boltzmann
  Shannon entropy $\mathcal{E}(\rho) = \int \rho\log\rho dx$,
  \begin{align*}
    \d_t \rho = \nabla \cdot \left(\int_{\R^3}K(x,y) \nabla \frac{\delta \mathcal{E}}{\delta\rho}(y) dy\right)~~\text{with}~~ K(x,y)= \delta_{\{x=y\}} \rho(x) \mathcal{L}\rho(x) - \frac{\rho(x)\rho(y)}{4\pi|x-y|}.
  \end{align*}
\end{thm}
Based on this structure, we can define a distance function $W_K$ (\ref{wk}) in the Benamou-Brenier
fashion, and have a lower bound with respect to the Wasserstein-$1$ distance (see Section 2.1). The
corresponding geodesic equations can also be computed from the Hamiltonian (in Section 2.2).

In \cite{carrillo2020,carrillo2019}, the gradient flow structure for the classical Landau equation can be written as
\begin{align*}
          \d_t \rho  =   \nabla \cdot \left(\int_{\R^3}\rho(x)\rho(y)|x-y|^{2+\gamma}\Pi[x-y](\nabla \log \rho(x)-\nabla\log \rho(y)) dy\right),
\end{align*}
with the projection matrix $\Pi[z] = I-\frac{z\otimes z}{|z|^2}$. It is easy to check that, when $\gamma = -d= -3$ and $\pi[z]$ is set to be $1$, the above two gradient flow structures are equivalent. 

Let us denote $\rho_t \equiv \rho$ and $K_{\rho_t} \equiv K$ to emphasize the density path. The
second main result of this note provides a time-dependent convergence rate result for the entropy
functional by estimating its Hessian operator. The proof essentially follows the Bakry-Emery
strategy and assumes that $\rho_t$ is of sufficient regularity throughout the computations, which has been proven true for radially symmetric solutions.
\begin{thm}\label{convergencerate}
  Along the gradient flow (\ref{landau}), equipped with $\Phi_t = -\log \rho_t$, we can compute the Riemannian Hessian operator of the entropy as
  \begin{align}
    \frac{d^2}{dt^2}\mathcal{E}(\rho_t)  =& \nonumber -\frac{3}{2}
  \int \rho_t^2(-\Delta)^{-1} \rho_t|\nabla \Phi_t|^2 dx + \int \rho_t^2\nabla \Phi_t (-\Delta)^{-1}(\rho_t \nabla \Phi_t ) dx\\
  &-\frac{1}{4}\int \nabla \rho_t \nabla ((-\Delta)^{-1} \rho_t)^2|\nabla \Phi_t|^2 dx+\int \rho_t ((-\Delta)^{-1}\rho_t  )^2 ||\nabla^2 \Phi_t||^2 dx\label{convrate}.
  \end{align}
  With an additional assumption that if there exists $\gamma\in (0,1/7)$ such that 
  \begin{equation}\label{eqnpos}
     \int\left( 2\gamma\rho_t - \frac{|\nabla\rho_t|^2(-\Delta)^{-1}\rho_t}{\rho_t^2}\right)\left(\rho_t \nabla \Phi_t (-\Delta)^{-1}(\rho_t \nabla \Phi_t )+3\rho_t^2\right) dx\ge 0,
  \end{equation}
 we then have the convergence rate for the entropy
  \begin{align*}
    \frac{d}{dt}\mathcal{E}(\rho_t) \leq -\alpha \int_{t}^{\infty} \int \rho_t^3 dx dt,
  \end{align*}
 where $\alpha\in(0,1)$ is a constant depending on $\gamma$.
\end{thm}

This result provides another view comparing to the dissipation-Fisher information relation
(\ref{prevrel}). The detailed computations, which heavily use the geodesic equations, the relation
$\Phi_t = -\log\rho_t$ along the gradient flow, and the Bochner's formula, will be given in Section 3. We would like to point out that it is possible to carry out similar computations for the original Landau equation. However, due to the existence of the projection matrix $\Pi[z]$ and the different steady solution, the computations can be significantly more complicated and the convergence result might also change.

\paragraph{Organization.}
The rest of the note is organized as follows. Section 2 details the gradient flow structure and
Section 3 gives the proof of Theorem \ref{convergencerate}.

\section{The gradient flow structure}

Consider the density space (or sometimes called the statistical manifold)
\begin{align}
  \mathcal{M} = \{ \text{non-negative functions } \rho \in \R^3  \text{ and } \int_{\R^3} \rho =1\}.
\end{align}
The tangent space of $\mathcal{M}$ at $\rho\in \mathcal{M}$ is given by
\begin{align}
  \mathcal{T}_{\rho}\mathcal{M} = \{\text{functions } \sigma\in \R^3 \text{ and }\int_{\R^3} \sigma = 0 \}.
\end{align}
The key object of this note is the nonlocal metric tensor defined as follows.

\begin{definition}{(Nonlocal metric tensor)}
  Given $\rho \in \mathcal{M}$, for $\sigma_{1,2}\in \mathcal{T}_{\rho}\mathcal{M}$, the 
  nonlocal metric tensor $g_{\mathcal{K}}$ is given by
  \begin{align}\label{metric1}
    g_{\mathcal{K}} (\sigma_1, \sigma_2) :=\langle \sigma_1, (-\mathcal{K})^{-1} \sigma_2)\rangle =
    \langle \Phi_1, -\mathcal{K} \Phi_2 \rangle,
  \end{align}
  where
  \begin{align}\label{operator}
    \mathcal{K} u(x) = \nabla \cdot \left(\int_{\R^3}K(x,y) \nabla u(y) dy\right)~~\text{with}~~ K(x,y)=
    \delta_{\{x=y\}} \rho(x) \mathcal{L}\rho(x) - \frac{\rho(x)\rho(y)}{4\pi|x-y|},
  \end{align}
  and $\Phi_i$ is a weak solution to the equation
  \begin{align}\label{sig1}
    \sigma_{i}(x) = -\nabla \cdot\left( \int_{\R^3} K(x,y)\nabla \Phi_i(y)dy\right) = -\mathcal{K} \Phi_i (x),~~ i = 1,2.
  \end{align}
\end{definition}

To show that the metric tensor $g_{\mathcal{K}}$ is well-defined, one needs to verify that it is
bilinear, symmetric, and positive semi-definite. The first two conditions can be checked directly
while the last condition requires that for any $u$ in some Banach space
\begin{align}\label{pos}
\iint \frac{\rho(x)\rho(y) u^2(x)}{|x-y|} - \frac{\rho(x)\rho(y)u(x) u(y)}{|x-y|} dxdy \geq 0.
\end{align}
This inequality holds by simply using the symmetry and Young's inequality,
\begin{align*}
  \iint \frac{\rho(x)\rho(y) u^2(x)}{|x-y|} dxdy & = \iint \frac{\rho(x)\rho(y)
    (u^2(x)/2+u^2(y)/2)}{|x-y|} dxdy \\
  & \geq \iint \frac{\rho(x)\rho(y)u(x) u(y)}{|x-y|} dxdy.
\end{align*}

The following theorem states that the isotropic Landau equation is the gradient flow of the entropy
with respect to the Riemannian structure introduced above.
\begin{thm}
  Given the Boltzmann Shannon entropy $\mathcal{E}: \mathcal{M}\to \R$ where
  \begin{align}\label{entropy}
    \mathcal{E}(\rho) = \int_{\R^3} \rho \log\rho dx,
  \end{align}
  under the nonlocal metric tensor defined in (\ref{metric1}), the gradient flow dynamics of (\ref{entropy}) is exactly (\ref{landau}).
\end{thm}
\begin{proof}
  Note that
  \begin{align*}
    g_{\mathcal{K}} (\grad\mathcal{E}|_{\rho}, \sigma) = \int_{\R^3}  \frac{\delta\mathcal{E}}{\delta \rho}(x) \sigma(x) dx.
  \end{align*}
  By the definition (\ref{metric1}),
  \begin{align*}
    g_{\mathcal{K}} (\grad\mathcal{E}|_{\rho}, \sigma) = \int_{\R^3}
    \grad\mathcal{E}|_{\rho}(x) (-\mathcal{K})^{-1} \sigma dx = \int_{\R^3}
    \grad\mathcal{E}|_{\rho}(x) \Phi(x) dx.
  \end{align*}
  Plugging in the equation (\ref{sig1}) and using
  $\nabla\frac{\delta\mathcal{E}}{\delta\rho}=\frac{\nabla \rho}{\rho}$ leads to
  \begin{align*}
    \int_{\R^3}  \frac{\delta\mathcal{E}}{\delta \rho}(x) \sigma(x) dx &= -\int_{\R^3} \frac{\delta\mathcal{E}}{\delta \rho}(x) \nabla \cdot \bigg(\int_{\R^3} \bigg( \delta_{\{x=y\}} \rho(x)\mathcal{L}\rho(x) - \frac{\rho(x)\rho(y)}{4\pi|x-y|}\bigg) \nabla \Phi(y)dy\bigg) dx\\
    &= \int_{\R^3}\bigg( \int_{\R^3} \bigg( \delta_{\{x=y\}} \rho(x)\mathcal{L}\rho(x) - \frac{\rho(x)\rho(y)}{4\pi|x-y|}\bigg)\nabla \frac{\delta\mathcal{E}}{\delta \rho}(x) dx \bigg)\nabla \Phi(y) dy\\
    & = \int_{\R^3} \bigg(\nabla \rho(y)\mathcal{L}\rho(y) - \rho(y)\mathcal{L}\nabla\rho(y) \bigg)\nabla \Phi(y) dy\\
    & =-\int_{\R^3} \nabla \cdot\bigg(\mathcal{L}\rho(x)\nabla \rho(x) - \rho(x)\nabla\mathcal{L}\rho(x) \bigg)\Phi(x)dx.
  \end{align*}
  Therefore,
  \begin{align*}
    \grad\mathcal{E}|_{\rho}(x)  =  -\nabla \cdot\big(\mathcal{L}\rho(x)\nabla \rho(x) - \rho(x)\nabla\mathcal{L}\rho(x) \big).
  \end{align*}
  Since $\d_t \rho = -\grad\mathcal{E}|_{\rho}$, the Riemannian gradient flow in
  $(\mathcal{M}, g_{\mathcal{K}})$ gives the isotropic Landau equation as desired.
\end{proof}

Based on the above characterization, it is natural to define a Benamou-Brenier like formalism
\cite{benamou2000computational} of the distance similar to the classical Wasserstein distance, but
instead with a nonlocal mobility defined in (\ref{operator}).
\begin{definition}
  If the kernel $K(x,y)$ is well-defined, the distance function $W_{K}:\mathcal{M}\times
  \mathcal{M}\to \R_+$ between two functions $\rho_0(x) = \rho(x,0)$ and $\rho_1(x)= \rho(x,1)$ is
  \begin{align}
    \label{wk}
    W_{K} (\rho_0,\rho_1):=\inf_{v,\rho} \bigg( \int_0^1\int_{\R^3}\int_{\R^3} v(x,t)K(x,y)v(y,t)dxdy dt \bigg)^{1/2},
  \end{align}
  and the infimum is taken over all smooth paths $\rho: \R^3 \times[0,1] \to \R_+$ and vector field
  $v:\R^3 \times[0,1] \to \mathcal{T}\R^3$ satisfying the continuity equation
  \begin{align}
    \label{cteqn}
    \d_t \rho(x,t) + \nabla\cdot \left(\int_{\R^3} K(x,y) v(y,t) dy\right) = 0.
  \end{align}
\end{definition}

\subsection{Comparison to $L^1-$Wasserstein distance}

Based on the conditional regularity results provided in \cite{gualdani2018global,
  gualdani2018review}, we can obtain a lower bound for the distance introduced in (\ref{wk}) in
terms of the $L^1-$Wasserstein distance. Let us first recall that the $L^1-$Wasserstein distance
between $\rho_0, \rho_1 \in \mathcal{M}$ is defined as
\begin{align}\label{l1w}
  W_1(\rho_0,\rho_1): = \inf_{\pi\in \Gamma(\rho_0,\rho_1)} \int|x-y| \pi(dx, dy),
\end{align}
where $\Gamma(\rho_0,\rho_1)$ is the set of all couplings of $\rho_0$ and $\rho_1$.

\begin{thm}
  If $\rho_0$ is even, $\rho_0\log\rho_0\in L^1(\R^3)$ and satisfies the $\varepsilon-$Poincar{\'e}
  inequality (\ref{poincare}), then we have the bound
  \begin{align}\label{aug11}
    W_1(\rho_0,\rho_1) \leq C W_{K}(\rho_0,\rho_1).
  \end{align}
\end{thm}
\begin{proof}
  Let $\varphi: \R^3\to \R $ be a bounded $1-$Lipschitz function. Clearly $\varphi\in
  W^{1,\infty}(\R^3)$. Using the continuity equation (\ref{cteqn}) and integration by parts gives
  rise to
  \begin{align*}
    &\left|\int \varphi \rho_1 dx - \int \varphi \rho_0 dx \right| = \left|\int_0^1 \int \varphi \d_t \rho ~ dx dt\right| =
    \left|\int_0^1 \int_{\R^3\times \R^3}\nabla \varphi(x) K(x,y) v(y,t)~ dx dy dt\right| \\
    &\leq \bigg(\int_0^1  \int_{\R^3\times \R^3}\nabla \varphi(x) K(x,y)\nabla \varphi(y) ~ dxdy dt\bigg)^{\frac{1}{2}} \bigg(\int_0^1  \int_{\R^3\times \R^3}v(x,t) K(x,y) v(y,t) ~ dxdy dt\bigg)^{\frac{1}{2}}.
  \end{align*}
  The last inequality uses Cauchy-Schwarz since (\ref{pos}) holds. Now as $||\nabla
  \varphi||_{L^\infty} = 1$, using the uniform bound (\ref{fracbd}) we have that
  \begin{align*}
    \left| \int_0^1  \int_{\R^3\times \R^3}\nabla \varphi(x) K(x,y)\nabla \varphi(y) ~ dxdy dt \right| &\leq 2\int_0^1 \int_{\R^3} \rho(x,t)|| (-\Delta)^{-1}\rho(x,t) ||_{L^{\infty}(t,1;\R^3)} dx dt \\
    &\leq 2 C( n) \int_0^1 \bigg(\frac{1}{t}+1\bigg)^{s} dt \leq C.
  \end{align*}
  Taking the supremum over all bounded $1-$Lipschitz functions $\varphi$ on the left hand side
  accompanied with Kantorovich-Rubinstein duality, and taking the infimum over $v, \rho$ on the
  right hand side, we then obtain the inequality (\ref{aug11}).
\end{proof}

\subsection{The geodesic equations}

Let us consider the geometric action functional in the density space 
\begin{align}\label{action}
  \mathcal{L}(\rho_t, \d_t \rho_t) =\frac{1}{2} \int_0^1 \int  \d_t \rho_t (-\mathcal{K}_{\rho_t})^{-1} \d_t \rho_t dxdt,
\end{align}
where $\rho_t = \rho(x,t)$ is the density path connecting $\rho_0 $ and $\rho_1$,
$\mathcal{K}_{\rho_t}$ is the Onsager operator defined in (\ref{operator}) with its dependency on
$\rho_t$ explicitly written.

\begin{lemma}\label{geodeqn}
  With the relation $\Phi_t = (-\mathcal{K}_{\rho_t})^{-1} \d_t \rho_t$, the geodesic
  equations are
  \begin{equation} \label{geoeqn2}
    \left\{
    \begin{array}{l}
      \d_t\rho_t + \mathcal{K}_{\rho_t}\Phi_t = 0\\
      \d_t \Phi_t + \frac{1}{2}\big(|\nabla \Phi_t|^2 (-\Delta)^{-1}\rho_t + (-\Delta)^{-1}(|\nabla\Phi_t|^2 \rho_t)\big) - \nabla \Phi_t (-\Delta)^{-1}(\rho_t \nabla \Phi_t) = 0
    \end{array}
    \right.
  \end{equation}
\end{lemma}

\begin{proof}
The derivation follows directly from the Hamiltonian formulation, which by Legendre transform is
\begin{align}
  \mathcal{H}(\rho_t, \Phi_t) = \sup_{\Phi_t \in C^{\infty}(\mathcal{M})} \int \Phi_t \d_t\rho_t dx - \mathcal{L}(\rho_t, \d_t \rho_t).
\end{align}
The supremum is obtained when $\Phi_t = (-\mathcal{K}_{\rho_t})^{-1} \d_t \rho_t$, and the
Hamiltonian is
\begin{align}\label{8221}
  \mathcal{H}(\rho_t, \Phi_t) = \frac{1}{2}\int \rho_t (-\mathcal{K}_{\rho_t})^{-1} \rho_t dx =
  \frac{1}{2}\iint \nabla \Phi_t(x) K_{\rho_t} (x,y) \nabla \Phi_t(y) dxdy.
\end{align}
The co-geodesic flow satisfies
\begin{align}\label{822}
  \d_t\rho_t = \frac{\delta \mathcal{H}(\rho_t, \Phi_t) }{\delta\Phi_t}, ~~~ \d_t \Phi_t  = - \frac{\delta \mathcal{H}(\rho_t, \Phi_t)}{\delta \rho_t}.
\end{align}
The first equation of \eqref{geoeqn2} can be easily obtained from the first relation in (\ref{822}).
To obtain the second equation of \eqref{geoeqn2} from the second relation in (\ref{822}) here, we
write (\ref{8221}) as
\begin{align*}
  \mathcal{H}(\rho_t, \Phi_t) &=\frac{1}{2}\iint  \nabla \Phi_t(x)\bigg( \delta_{\{x=y\}}\rho_t(x)(-\Delta)^{-1}\rho_t(x) - \frac{\rho_t(x)\rho_t(y)}{4\pi|x-y|}\bigg) \nabla \Phi_t(y) dxdy\\
  &= \frac{1}{2} \int | \nabla \Phi_t(x)|^2 \rho_t(x) (-\Delta)^{-1}\rho_t(x) dx - \frac{1}{2}\iint  \nabla \Phi_t(x) \frac{\rho_t(x)\rho_t(y)}{4\pi|x-y|}\nabla \Phi_t(y) dxdy.
\end{align*}
Thus for any $v\in \mathcal{M}$,
\begin{align*}
  & \int \frac{\delta \mathcal{H}}{\delta \rho_t} v dx = \frac{d}{d\epsilon} \mathcal{H}(\rho_t + \epsilon v, \Phi_t)|_{\epsilon = 0} \\
  =& \frac{1}{2}\int  | \nabla \Phi_t(x)|^2 (v(x)(-\Delta)^{-1}\rho_t(x) + \rho_t(x) (-\Delta)^{-1}v(x)) dx\\
  &-\frac{1}{2} \iint \nabla \Phi_t(x) \frac{\rho_t(x)v(y) + v(x)\rho_t(y)}{4\pi|x-y|}\nabla \Phi_t(y) dxdy\\
  =&\frac{1}{2}\int \bigg(|\nabla \Phi_t|^2 (-\Delta)^{-1}\rho_t + (-\Delta)^{-1}(|\nabla\Phi_t|^2 \rho_t)\bigg) v dx -
  \int \nabla \Phi_t(x) (-\Delta)^{-1}(\rho_t \nabla \Phi_t) v dx,
\end{align*}
which gives the second equation that we stated.
\end{proof}

\section{Estimate of the Hessian operator}

This section is devoted to the proof of Theorem \ref{convergencerate}. The Hessian operator of the
entropy can be computed by taking the second time derivative of $\mathcal{E}$ along the geodesic
equations as in Lemma \ref{geodeqn}. Let us use $\delta \mathcal{E}$ to denote $\frac{\delta
  \mathcal{E}}{\delta \rho}$ for convenience. Note that
\begin{align*}
  \frac{d}{dt}\mathcal{E}(\rho_t) = \langle \grad\mathcal{E}|_{\rho_t}, \Phi_t \rangle &=\int \delta \mathcal{E} \d_t \rho_t dx \\
  &= \iint \nabla \delta \mathcal{E} (\rho_t(x))K_{\rho_t} (x,y) \nabla \Phi_t(y) dxdy.
\end{align*}
The second variation of $\mathcal{E}$ is
\begin{align*}
  \frac{d^2}{dt^2}\mathcal{E}(\rho_t) = \langle \Hess\mathcal{E}|_{\rho_t}\Phi_t, \Phi_t\rangle =&\int \nabla \frac{d}{dt} \delta \mathcal{E}(\rho_t(x))  K_{\rho_t} (x,y) \nabla \Phi_t(y) dxdy\\
  &+ \int \nabla \delta \mathcal{E} (\rho_t(x)) \d_t K_{\rho_t} (x,y) \nabla \Phi_t(y) dxdy\\
  &+\int \nabla \delta \mathcal{E} (\rho_t(x))K_{\rho_t} (x,y) \nabla \d_t \Phi_t(y) dxdy\\
  =& I+II+III.
\end{align*}

We begin the proof of Theorem \ref{convergencerate}. The following rather long computations will
involve the quantity
\begin{align}\label{oct281}
  \nonumber \d_t \rho_t =&  - \nabla \cdot \big( \rho_t (-\Delta)^{-1} \rho_t \nabla \Phi_t- \rho_t (-\Delta)^{-1}( \rho_t \nabla \Phi_t) \big)\\
  \nonumber=& - \nabla \rho_t (-\Delta)^{-1} \rho_t \nabla \Phi_t-  \rho_t (-\Delta)^{-1} \nabla \rho_t \nabla \Phi_t - \rho_t (-\Delta)^{-1} \rho_t \Delta \Phi_t \\
  &+\nabla \rho_t (-\Delta)^{-1} ( \rho_t \nabla \Phi_t ) + \rho_t^2.
\end{align}
The last term above uses the relation $\Phi_t = -\log\rho_t$ since we follow the Hessian operator
along the gradient flow. Thus
\begin{align}\label{oct282}
  \nabla\cdot(-\Delta)^{-1}(\rho_t\nabla \Phi_t) = \rho_t.
\end{align}
However, unless we are to analyze some difficult terms, the notation $\Phi_t$ will be kept for the majority of the
computation in order to explore the associated Hessian structure.

For the first term $I$, let us use (\ref{oct281}) and study its quadratic expansion,
\begin{align*}
  I =& \int \delta^2 \mathcal{E}(\rho_t(x)) \left(\nabla \cdot \left( \int K_{\rho_t} (x,y) \nabla \Phi_t(y) dy\right) \right)^2 dx\\
  =& \int \frac{1}{\rho_t(x)} \bigg(\nabla \cdot (\rho_t(x)(-\Delta)^{-1} \rho_t(x) \nabla \Phi_t (x) - \rho_t(x) (-\Delta)^{-1}(\rho_t(x) \nabla \Phi_t(x)))\bigg)^2 dx\\
  =&\int \frac{|\nabla \rho_t|^2}{\rho_t}((-\Delta)^{-1}\rho_t)^2 |\nabla \Phi_t|^2 dx + \int \rho_t ((-\Delta)^{-1}\nabla \rho_t)^2 |\nabla \Phi_t|^2  dx+\int \rho_t ((-\Delta)^{-1}\rho_t)^2 (\Delta\Phi_t)^2 dx \\
  &+\int \frac{|\nabla \rho_t|^2}{\rho_t} |(-\Delta)^{-1} ( \rho_t \nabla \Phi_t )|^2 dx + \int \rho_t^3 dx + \int \nabla \rho_t \nabla ((-\Delta)^{-1} \rho_t)^2|\nabla \Phi_t|^2 dx \\
  &+ 2\int \nabla \rho_t ( (-\Delta)^{-1} \rho_t)^2 \nabla \Phi_t \Delta \Phi_t dx - 2\int \frac{|\nabla \rho_t|^2}{\rho_t}(-\Delta)^{-1} \rho_t \nabla \Phi_t(-\Delta)^{-1} ( \rho_t \nabla \Phi_t ) dx\\
  & + 2\int |\nabla \rho_t|^2 (-\Delta)^{-1} \rho_t  dx + \int \rho_t \nabla ((-\Delta)^{-1} \rho_t)^2 \nabla \Phi_t \Delta \Phi_t dx\\
  & - 2\int \nabla \rho_t (-\Delta)^{-1} \nabla \rho_t \nabla \Phi_t (-\Delta)^{-1} ( \rho_t \nabla \Phi_t ) dx - 2\int \rho_t^2(-\Delta)^{-1}\rho_t\Delta \Phi_t dx  \\
  & - 2\int \nabla \rho_t (-\Delta)^{-1} \rho_t \Delta \Phi_t  (-\Delta)^{-1} ( \rho_t \nabla \Phi_t ) dx.
\end{align*}

For the second term $II$,
\begin{align*}
  II =& \int \frac{\nabla \rho_t}{\rho_t}\bigg(\delta_{\{x=y\}} \d_t \rho_t (-\Delta)^{-1}\rho_t+\delta_{\{x=y\}}  \rho_t (-\Delta)^{-1}\d_t\rho_t - \frac{\d_t \rho_t(x) \rho_t(y)+ \rho_t(x)\d_t \rho_t(y)}{4\pi|x-y|}\bigg)\nabla \Phi_t(y) dxdy\\
  =&\int \frac{\nabla \rho_t}{\rho_t} \d_t\rho_t  (-\Delta)^{-1}\rho_t \nabla \Phi_t dx + \int \nabla \rho_t (-\Delta)^{-1}\d_t\rho_t \nabla \Phi_t dx \\
  &- \int \frac{\nabla \rho_t}{\rho_t} \d_t \rho_t (-\Delta)^{-1}(\rho_t \nabla \Phi_t )dx- \int \nabla \rho_t (-\Delta)^{-1}(\d_t\rho_t \nabla \Phi_t )dx
  := II_1 + II_2 +II_3 + II_4.
\end{align*}
Plugging in \eqref{oct281}, we have
\begin{align*}
  II_1 =& -  \int \frac{|\nabla \rho_t|^2}{\rho_t}((-\Delta)^{-1}\rho_t)^2 |\nabla \Phi_t|^2 dx  -\frac{1}{2} \int \nabla \rho_t \nabla ((-\Delta)^{-1} \rho_t)^2|\nabla \Phi_t|^2 dx \\
  &- \int \nabla \rho_t ((-\Delta)^{-1} \rho_t)^2 \nabla \Phi_t \Delta \Phi_t dx + \int \frac{|\nabla \rho_t|^2}{\rho_t}  (-\Delta)^{-1}\rho_t \nabla \Phi_t(-\Delta)^{-1} ( \rho_t \nabla \Phi_t )  dx\\
  &- \int |\nabla \rho_t|^2  (-\Delta)^{-1}\rho_t   dx,
\end{align*}
and
\begin{align*}
  II_3 =&\int \frac{|\nabla \rho_t|^2}{\rho_t}(-\Delta)^{-1} \rho_t \nabla \Phi_t (-\Delta)^{-1} (\rho_t \nabla \Phi_t ) dx +\int  \nabla \rho_t(-\Delta)^{-1} \nabla \rho_t \nabla \Phi_t (-\Delta)^{-1} (\rho_t \nabla \Phi_t ) dx\\
  &+\int \nabla \rho_t (-\Delta)^{-1} \rho_t \Delta \Phi_t  (-\Delta)^{-1} ( \rho_t \nabla \Phi_t ) dx -\int \frac{|\nabla \rho_t|^2}{\rho_t} |(-\Delta)^{-1} ( \rho_t \nabla \Phi_t )|^2 dx+\frac{1}{2} \int \rho_t^3 dx.
\end{align*}
Moreover,
\begin{align*}
  II_4 =& -\int (-\Delta)^{-1}\nabla \rho_t \nabla \Phi_t \d_t \rho_t dx \\
  =& \frac{1}{2}\int \nabla \rho_t \nabla ((-\Delta)^{-1} \rho_t)^2|\nabla \Phi_t|^2 dx + \int \rho_t((-\Delta)^{-1}\nabla \rho_t)^2 |\nabla \Phi_t|^2  dx + \frac{1}{2}\int \rho_t \nabla ((-\Delta)^{-1} \rho_t )^2\nabla \Phi_t \Delta \Phi_t dx\\
  & - \int \nabla \rho_t  (-\Delta)^{-1} \nabla \rho_t \nabla \Phi_t (-\Delta)^{-1} ( \rho_t \nabla \Phi_t ) dx +\frac{1}{2} \int \rho_t^3 dx.
\end{align*}

For the third term, we use the geodesic equation in (\ref{geodeqn}) and note that the (\ref{landau})
also can be written as $\d_t\rho_t = (-\Delta)^{-1} \rho_t \Delta \rho_t + \rho_t^2$. Therefore,
\begin{align*}
  III =& \int \nabla \cdot \bigg(\int \big(\delta_{\{x=y\}}\rho_t(-\Delta)^{-1}\rho_t - \frac{\rho_t(x)\rho_t(y)}{4\pi|x-y|} \big) \frac{\nabla \rho_t(y)}{\rho_t(y)} dy\bigg)\\ &\times\bigg(\frac{1}{2}\big(|\nabla \Phi_t|^2 (-\Delta)^{-1}\rho_t + (-\Delta)^{-1}(|\nabla\Phi_t|^2 \rho_t)\big) - \nabla \Phi_t (-\Delta)^{-1}(\rho_t \nabla \Phi_t) \bigg) dx\\
  =& \int  ((-\Delta)^{-1} \rho_t \Delta \rho_t  + \rho_t^2) \bigg(\frac{1}{2}\big(|\nabla \Phi_t|^2 (-\Delta)^{-1}\rho_t + (-\Delta)^{-1}(|\nabla\Phi_t|^2 \rho_t)\big) - \nabla \Phi_t (-\Delta)^{-1}(\rho_t \nabla \Phi_t) \bigg) dx\\
  =&\frac{1}{2} \int ((-\Delta)^{-1} \rho_t)^2 \Delta\rho_t |\nabla \Phi_t|^2 dx + \frac{1}{2}\int (-\Delta)^{-1}((-\Delta)^{-1} \rho_t \Delta \rho_t  + \rho_t^2) \rho_t  |\nabla \Phi_t|^2 dx\\
  & - \int (-\Delta)^{-1} \rho_t \Delta \rho_t \nabla \Phi_t (-\Delta)^{-1}(\rho_t \nabla \Phi_t ) dx + \frac{1}{2} \int \rho_t^2 (-\Delta)^{-1}\rho_t  |\nabla \Phi_t|^2 dx \\
  &- \int \rho_t^2\nabla \Phi_t (-\Delta)^{-1}(\rho_t \nabla \Phi_t ) dx  := III_1+III_2+III_3+III_4+III_5.
\end{align*}

Combining $I, II_1, II_3$ and $II_4$ results in
\begin{align*}
  I+&II_1+II_3+II_4 \\
  =&  2\int \rho_t ((-\Delta)^{-1}\nabla \rho_t)^2 |\nabla \Phi_t|^2  dx+\int \rho_t ((-\Delta)^{-1}\rho_t)^2 (\Delta\Phi_t)^2 dx  + 2\int \rho_t^3 dx \\
  &+ \int \nabla \rho_t \nabla ((-\Delta)^{-1} \rho_t)^2|\nabla \Phi_t|^2 dx + \int \nabla \rho_t ( (-\Delta)^{-1} \rho_t)^2 \nabla \Phi_t \Delta \Phi_t dx \\
  & + \int |\nabla \rho_t|^2 (-\Delta)^{-1} \rho_t dx + \frac{3}{2}\int \rho_t \nabla ((-\Delta)^{-1} \rho_t)^2 \nabla \Phi_t \Delta \Phi_t dx\\
  & - 2\int \nabla \rho_t (-\Delta)^{-1} \nabla \rho_t \nabla \Phi_t (-\Delta)^{-1} ( \rho_t \nabla \Phi_t ) dx - 2\int \rho_t^2  (-\Delta)^{-1} \rho_t \Delta \Phi_t dx\\
  & - \int \nabla \rho_t (-\Delta)^{-1} \rho_t \Delta \Phi_t  (-\Delta)^{-1} ( \rho_t \nabla \Phi_t ) dx  := IV.
\end{align*}
We can now rearrange the first three lines in $IV$ in a nicer way by doing integration by parts,
\begin{align*}
\int \nabla \rho_t ( (-\Delta)^{-1} \rho_t)^2 \nabla \Phi_t \Delta \Phi_t dx =& -\int \rho_t \nabla ( (-\Delta)^{-1} \rho_t)^2 \nabla \Phi_t \Delta \Phi_t dx - \int \rho_t ( (-\Delta)^{-1} \rho_t)^2 (\Delta \Phi_t)^2 dx\\
&- \int \rho_t  ( (-\Delta)^{-1} \rho_t)^2 (\nabla \Phi_t, \nabla \Delta \Phi_t) dx,
\end{align*}
and as a result 
\begin{align*}
    IV =& 2\int \rho_t ((-\Delta)^{-1}\nabla \rho_t)^2 |\nabla \Phi_t|^2  dx + 2\int \rho_t^3 dx + \int \nabla \rho_t \nabla ((-\Delta)^{-1} \rho_t)^2|\nabla \Phi_t|^2 dx \\
& - \int \rho_t  ( (-\Delta)^{-1} \rho_t)^2 (\nabla \Phi_t, \nabla \Delta \Phi_t) dx + \frac{1}{2}\int \rho_t \nabla ((-\Delta)^{-1} \rho_t)^2 \nabla \Phi_t \Delta \Phi_t dx\\
& - 2\int \nabla \rho_t (-\Delta)^{-1} \nabla \rho_t \nabla \Phi_t (-\Delta)^{-1} ( \rho_t \nabla \Phi_t ) dx - 2\int \rho_t^2  (-\Delta)^{-1} \rho_t \Delta \Phi_t dx\\
& - \int \nabla \rho_t (-\Delta)^{-1} \rho_t \Delta \Phi_t  (-\Delta)^{-1} ( \rho_t \nabla \Phi_t ) dx + \int |\nabla \rho_t|^2 (-\Delta)^{-1} \rho_t dx.
\end{align*}
Note that
\begin{align*}
III_1 =& \frac{1}{2}\int  \rho_t ((-\Delta)^{-1} \rho_t)^2 \Delta (\nabla \Phi_t, \nabla \Phi_t)dx + \frac{1}{2}\int  \rho_t \Delta ((-\Delta)^{-1} \rho_t)^2 |\nabla \Phi_t|^2 dx \\
&+ \int \rho_t \nabla((-\Delta)^{-1} \rho_t)^2 \nabla |\nabla \Phi_t|^2 dx.
\end{align*}
Using the Bochner's formula
\begin{align}
\frac{1}{2} \Delta (\nabla \Phi_t, \nabla \Phi_t) - (\nabla \Phi_t, \nabla \Delta \Phi_t) = ||\nabla^2 \Phi_t||^2,
\end{align}
one can combine the first term in $III_1$ and the first term in the second line in $IV$ to obtain
\begin{align*}
  IV+III_1 =& 2\int \rho_t ((-\Delta)^{-1}\nabla \rho_t)^2 |\nabla \Phi_t|^2  dx + 2\int \rho_t^3 dx + \int \nabla \rho_t \nabla ((-\Delta)^{-1} \rho_t)^2|\nabla \Phi_t|^2 dx \\
& +\int \rho_t ( (-\Delta)^{-1}\rho_t  )^2 ||\nabla^2 \Phi_t||^2 dx+  \frac{1}{2}\int \rho_t \nabla ((-\Delta)^{-1} \rho_t)^2 \nabla \Phi_t \Delta \Phi_t dx\\
& - 2\int \nabla \rho_t (-\Delta)^{-1} \nabla \rho_t \nabla \Phi_t (-\Delta)^{-1} ( \rho_t \nabla \Phi_t ) dx - 2\int \rho_t^2  (-\Delta)^{-1} \rho_t \Delta \Phi_t dx\\
& - \int \nabla \rho_t (-\Delta)^{-1} \rho_t \Delta \Phi_t  (-\Delta)^{-1} ( \rho_t \nabla \Phi_t ) dx + \int |\nabla \rho_t|^2 (-\Delta)^{-1} \rho_t dx\\
&+ \frac{1}{2}\int  \rho_t \Delta ((-\Delta)^{-1} \rho_t)^2 |\nabla \Phi_t|^2 dx + \int \rho_t \nabla((-\Delta)^{-1} \rho_t)^2 \nabla |\nabla \Phi_t|^2 dx.
\end{align*}
We continue to apply the integration by parts here. Note that the first term in the fourth line
above can be written as
\begin{align*}
    - \int \nabla \rho_t (-\Delta)^{-1} &\rho_t \Delta \Phi_t  (-\Delta)^{-1} ( \rho_t \nabla \Phi_t ) dx = \int \Delta \rho_t \nabla \Phi_t(-\Delta)^{-1} \rho_t(-\Delta)^{-1} ( \rho_t \nabla \Phi_t ) dx\\
    &+\int \nabla \rho_t \nabla \Phi_t(-\Delta)^{-1}\nabla \rho_t(-\Delta)^{-1} ( \rho_t \nabla \Phi_t ) dx - \int|\nabla \rho_t|^2(-\Delta)^{-1}  \rho_t  dx,
\end{align*}
with the last term obtained by plugging in $\Phi_t = -\log \rho_t$. Furthermore, using this substitution can reformulate the second term in the third line of $IV+III_1$ as
\begin{align*}
  - 2\int \rho_t^2  (-\Delta)^{-1} \rho_t \Delta \Phi_t dx = 2\int \rho_t\Delta\rho_t (-\Delta)^{-1} \rho_t dx - 2\int |\nabla \rho_t|^2(-\Delta)^{-1}  \rho_t  dx.
\end{align*}
Finally, we can rearrange the terms into the form
\begin{align*}
  IV+III_1+III_3 =&2\int \rho_t ((-\Delta)^{-1}\nabla \rho_t)^2 |\nabla \Phi_t|^2  dx - \int \nabla \rho_t (-\Delta)^{-1} \nabla \rho_t \nabla \Phi_t (-\Delta)^{-1} ( \rho_t \nabla \Phi_t ) dx\\ &+\int \rho_t ((-\Delta)^{-1}\rho_t  )^2 ||\nabla^2 \Phi_t||^2 dx+ 2\int \rho_t^3 dx   + 2\int \rho_t\Delta\rho_t (-\Delta)^{-1} \rho_t dx\\
&+ \frac{1}{2}\int  \rho_t \Delta ((-\Delta)^{-1} \rho_t)^2 |\nabla \Phi_t|^2 dx + \int \rho_t \nabla((-\Delta)^{-1} \rho_t)^2 \nabla |\nabla \Phi_t|^2 dx\\
& + \int \nabla \rho_t \nabla ((-\Delta)^{-1} \rho_t)^2|\nabla \Phi_t|^2 dx + \frac{1}{2}\int \rho_t \nabla ((-\Delta)^{-1} \rho_t)^2 \nabla \Phi_t \Delta \Phi_t dx\\
&-2\int |\nabla \rho_t|^2(-\Delta)^{-1}
\rho_t dx:= V.
\end{align*}

Using $\Phi_t = -\log \rho_t$, the first line in $V$ can be reduced to
\begin{align*}
  2\int \rho_t ((-\Delta)^{-1}\nabla \rho_t)^2 |\nabla \Phi_t|^2  dx - \int \nabla \rho_t (-\Delta)^{-1} \nabla \rho_t \nabla \Phi_t (-\Delta)^{-1} ( \rho_t \nabla \Phi_t ) dx = \int \rho_t ((-\Delta)^{-1}\nabla \rho_t)^2|\nabla \Phi_t|^2 dx.
\end{align*}
Moreover, in the second line, we observe that
\begin{align*}
  2\int \rho_t^3 dx   &+ 2\int \rho_t\Delta\rho_t (-\Delta)^{-1} \rho_t dx = 2 \int \rho_t \d_t \rho_t dx = -2\int \rho_t \mathcal{K}_{\rho_t} \Phi_t dx\\
  &= -2 \int \rho_t^2 (-\Delta)^{-1}\rho_t  |\nabla \Phi_t|^2 dx +2 \int \rho_t^2\nabla \Phi_t (-\Delta)^{-1}(\rho_t \nabla \Phi_t ) dx  := VI.
\end{align*}
Therefore,
\begin{align*}
  VI+ III_4+III_5 = -\frac{3}{2} \int \rho_t^2 (-\Delta)^{-1}\rho_t  |\nabla \Phi_t|^2 dx + \int \rho_t^2\nabla \Phi_t (-\Delta)^{-1}(\rho_t \nabla \Phi_t ) dx .
\end{align*}
The third and fourth lines in $V$ can be organized as follows
\begin{align*}
    \frac{1}{2}\int & \rho_t \Delta ((-\Delta)^{-1} \rho_t)^2 |\nabla \Phi_t|^2 dx + \frac{5}{4}\int \rho_t \nabla((-\Delta)^{-1} \rho_t)^2 \nabla |\nabla \Phi_t|^2 dx + \int \nabla \rho_t \nabla ((-\Delta)^{-1} \rho_t)^2|\nabla \Phi_t|^2 dx \\
    &=-\frac{3}{4}\int  \rho_t \Delta ((-\Delta)^{-1} \rho_t)^2 |\nabla \Phi_t|^2 dx  - \frac{1}{4}\int \nabla \rho_t \nabla ((-\Delta)^{-1} \rho_t)^2|\nabla \Phi_t|^2 dx\\
    &=\frac{3}{2}\int \rho_t^2 (-\Delta)^{-1}\rho_t  |\nabla \Phi_t|^2 dx - \frac{3}{2}\int \rho_t  ((-\Delta)^{-1}\nabla \rho_t)^2|\nabla \Phi_t|^2 dx-\frac{1}{4}\int \nabla \rho_t \nabla ((-\Delta)^{-1} \rho_t)^2|\nabla \Phi_t|^2 dx.
\end{align*}
Collect all the terms above, we arrive at 
\begin{align*}
  V+III_4+III_5 =& - \frac{1}{2}\int \rho_t  ((-\Delta)^{-1}\nabla \rho_t)^2|\nabla \Phi_t|^2 dx-\frac{1}{4}\int \nabla \rho_t \nabla ((-\Delta)^{-1} \rho_t)^2|\nabla \Phi_t|^2 dx\\
  &+ \int \rho_t^2\nabla \Phi_t (-\Delta)^{-1}(\rho_t \nabla \Phi_t ) dx+\int \rho_t ((-\Delta)^{-1}\rho_t  )^2 ||\nabla^2 \Phi_t||^2 dx-2\int |\nabla \rho_t|^2(-\Delta)^{-1}
\rho_t dx,
\end{align*}
plus the remaining terms 
\begin{align*}
  II_2 + III_2 &= \int (-\Delta)^{-1}(\nabla \rho_t \nabla \Phi_t) \d_t\rho_t dx+ \frac{1}{2}\int (-\Delta)^{-1}((-\Delta)^{-1} \rho_t \Delta \rho_t  + \rho_t^2) \rho_t  |\nabla \Phi_t|^2 dx\\
  &= -\frac{1}{2}\int (-\Delta)^{-1} (\rho_t  |\nabla \Phi_t|^2)((-\Delta)^{-1} \rho_t \Delta \rho_t  + \rho_t^2) dx\\
  &=\frac{1}{2}
  \int \rho_t^2(-\Delta)^{-1} \rho_t|\nabla \Phi_t|^2 dx - \int\rho_t (-\Delta)^{-1} \nabla (\rho_t  |\nabla \Phi_t|^2)(-\Delta)^{-1} \nabla \rho_t dx,
\end{align*}
when using the integration by parts to move the Laplacian operator to other places. The second term above can be dealt with by viewing $\rho_t$ as
$-\nabla\cdot(-\Delta)^{-1}\nabla\rho_t$ and an integration by parts gives
\begin{align*}
  - \int\rho_t (-\Delta)^{-1} \nabla (\rho_t  |\nabla \Phi_t|^2)(-\Delta)^{-1} \nabla \rho_t dx = \frac{1}{2}\int \rho_t  ((-\Delta)^{-1}\nabla \rho_t)^2|\nabla \Phi_t|^2 dx.
\end{align*}
Now we are ready to wrap up all terms,
\begin{align*}
  I+II+III =& -\frac{3}{2}
  \int \rho_t^2(-\Delta)^{-1} \rho_t|\nabla \Phi_t|^2 dx + \int \rho_t^2\nabla \Phi_t (-\Delta)^{-1}(\rho_t \nabla \Phi_t ) dx\\
  &-\frac{1}{4}\int \nabla \rho_t \nabla ((-\Delta)^{-1} \rho_t)^2|\nabla \Phi_t|^2 dx+\int \rho_t ((-\Delta)^{-1}\rho_t  )^2 ||\nabla^2 \Phi_t||^2 dx.
\end{align*}
This proves \eqref{convrate}.

The term, $-\frac{1}{4}\int \nabla \rho_t \nabla ((-\Delta)^{-1} \rho_t)^2|\nabla \Phi_t|^2 dx$, can be rewritten using $\Phi_t = - \log \rho_t $ as
$$-\frac{1}{2}\int \rho_t \nabla \Phi_t \bigg( \frac{|\nabla\rho_t|^2(-\Delta)^{-1}\rho_t}{\rho_t^2}\bigg)(-\Delta)^{-1}(\rho_t \nabla \Phi_t ) dx.$$
We further decompose $\int \rho_t^2\nabla \Phi_t (-\Delta)^{-1}(\rho_t \nabla \Phi_t ) dx$ into
\begin{align*}
    \int \rho_t^2\nabla \Phi_t (-\Delta)^{-1}(\rho_t \nabla \Phi_t ) dx = \alpha \int \rho_t^3 dx + \beta \int \rho_t^3 dx + \gamma \int  \rho_t^2\nabla \Phi_t (-\Delta)^{-1}(\rho_t \nabla \Phi_t ) dx,
\end{align*}
with $2(\alpha+\beta)+\gamma=1$. Now reorganize terms, we obtain
\begin{align*}
    \beta \int \rho_t^3 dx-\frac{3}{2}
  \int \rho_t^2(-\Delta)^{-1} \rho_t|\nabla \Phi_t|^2 dx = \int \rho_t^2\bigg(\beta\rho_t -\frac{3}{2} \frac{|\nabla\rho_t|^2(-\Delta)^{-1}\rho_t}{\rho_t^2}\bigg) dx 
\end{align*}
and 
\begin{align*}
 \gamma\int \rho_t^2\nabla &\Phi_t (-\Delta)^{-1}(\rho_t \nabla \Phi_t ) dx-\frac{1}{4}\int \nabla \rho_t \nabla ((-\Delta)^{-1} \rho_t)^2|\nabla \Phi_t|^2 dx\\
  &= \frac{1}{2}\int \rho_t \nabla \Phi_t \bigg(2\gamma\rho_t - \frac{|\nabla\rho_t|^2(-\Delta)^{-1}\rho_t}{\rho_t^2}\bigg)(-\Delta)^{-1}(\rho_t \nabla \Phi_t ) dx. 
\end{align*}
With $\gamma=\beta/3 \in (0,1/7)$ for compatibility, then the positivity assumption \eqref{eqnpos}, with ignorance of the Hessian term, gives
\begin{align}\label{ineq1}
     \frac{d^2}{dt^2}\mathcal{E}(\rho_t) \geq \alpha \int \rho_t^3 dx.
\end{align}
Integrating (\ref{ineq1}) for $[t,\infty)$ results in
\begin{align*}
    \frac{d}{dt}\mathcal{E}(\rho_t)\leq - \alpha \int_{t}^{\infty} \int \rho_t^3 dx dt = -\alpha ||\rho_t||_{L^{3}([t,\infty);L^3(\R^3))}^3
\end{align*}
which concludes the proof of Theorem \ref{convergencerate}. This relation gives us another view of
time-dependent entropy dissipation comparable to (\ref{prevrel}).
\begin{remark} Recall the first variation of $\mathcal{E}$ is
\begin{align*}
  \frac{d}{dt}\mathcal{E}(\rho_t) = \langle \grad\mathcal{E}|_{\rho_t}, \Phi_t \rangle = -\iint \nabla \Phi_t(x) K_{\rho_t} (x,y) \nabla \Phi_t(y) dxdy.
\end{align*}
Ideally, one hopes to obtain the following inequality with some $\kappa(t)\ge 0$,
\begin{align*}
  \frac{d^2}{dt^2}\mathcal{E}(\rho_t)\ge -\kappa(t) \frac{d}{dt}\mathcal{E}(\rho_t),
\end{align*}
which can imply the convergence for the entropy with a rate depending on $\kappa(t)$. However, as
one can see from the rearrangement of $I+II+III$, we are not able to recover the full metric $\iint
\nabla \Phi_t(x) K_{\rho_t} (x,y) \nabla \Phi_t(y) dxdy$, although the terms with $\int \rho_t
\nabla \Phi_t (-\Delta)^{-1}(\rho_t \nabla \Phi_t ) dx$ are part of it. That is why we take the
assumption (\ref{eqnpos}) instead.
\end{remark}
\begin{remark}
Although the convergence rate estimate we provide here is a very crude bound, we can still observe
the slowness of the entropy decay rate. From Theorem 2 in \cite{gualdani2018review}, we can see that
$\rho_t$ decays asymptotically close to $1/t^{s}, t \gg 1$ (see (\ref{fracbd0})). It implies that
the entropy $\mathcal{E}(\rho_t)$ decreases at most polynomially fast when $t \gg 1$.
\end{remark}

\section*{Acknowledgments}
The authors thank Wuchen Li and Lenya Ryzhik for stimulating discussions. The first author is also grateful to Inwon Kim, Yao Yao, and Yuming Zhang for initial discussions. J.A. is supported by the
Joe Oliger Fellowship from Stanford University. The work of L.Y. is partially supported by the
National Science Foundation under award DMS-1818449.


\end{document}